\documentclass{amsart}

\usepackage{amssymb}

\usepackage[pagewise, displaymath, mathlines]{lineno}

\makeatletter
\let\my@abstract=\relax
\def\abstract#1{%
  \def\my@abstract{%
    \normalfont\Small
    \list{}{\labelwidth\z@
      \leftmargin3pc \rightmargin\leftmargin
      \listparindent\normalparindent \itemindent\z@
      \parsep\z@ \@plus\p@
      \let\fullwidthdisplay\relax
    }%
    \item[\hskip\labelsep\scshape\abstractname.]%
    #1
  \endlist}}
\def\@setabstracta{%
  \ifx\my@abstract\relax
  \else
    \skip@20\p@ \advance\skip@-\lastskip
    \advance\skip@-\baselineskip \vskip\skip@
  \my@abstract
    \prevdepth\z@ 
  \fi
}
\makeatother

\newtheorem{theorem}{Theorem}[section]
\newtheorem{lemma}[theorem]{Lemma}

\theoremstyle{definition}

\newtheorem{rem}[theorem]{Remark}

\theoremstyle{remark}

\numberwithin{equation}{section}

\newcommand{\Z}{\mathbb Z}
\newcommand{\Q}{\mathbb Q}
\newcommand{\R}{\mathbb R}
\newcommand{\C}{\mathbb C}

\input{xy}
\xyoption{all}

\begin{document}


\title[Witt equivalence of function fields]{Witt equivalence of function fields of curves over local fields}

\author{Pawe\l \ G\l adki}

\address{Institute of Mathematics,
University of Silesia, \newline \indent
ul. Bankowa 14, 40-007 Katowice, Poland \newline \indent
\and \newline \indent
Department of Computer Science,
AGH University of Science and Technology, \newline \indent
al. Mickiewicza 30, 30-059 Krak\'ow, Poland}
\email{pawel.gladki@us.edu.pl}

\author{\fbox{Murray Marshall}}

\thanks{Murray Marshall passed away in May 2015. Our community lost a brilliant mathematician and a wonderful man. We sorely miss him.}


\subjclass[2000]{Primary 11E81, 12J20 Secondary 11E04, 11E12}
\keywords{ symmetric bilinear forms, quadratic forms, Witt equivalence of fields, function fields, local fields, valuations, Abhyankar valuations}

\abstract{Two fields are Witt equivalent if their Witt rings of symmetric bilinear forms are isomorphic. Witt equivalent fields can be understood to be fields having the same quadratic form theory. The behavior of finite fields, local fields, global fields, as well as function fields of curves defined over archimedean local fields under Witt equivalence is well-understood. Numbers of classes of Witt equivalent fields with finite numbers of square classes are also known in some cases. Witt equivalence of general function fields over global fields was studied in the earlier work \cite{gm2016} by the authors, and applied to study Witt equivalence of function fields of curves over global fields. In this paper we extend these results to local case, i.e. we discuss Witt equivalence of function fields of curves over local fields. As an application, we show that, modulo some additional assumptions, Witt equivalence of two such function fields implies Witt equivalence of underlying local fields.}

\maketitle

\section{Introduction}

One of the fundamental problems in bilinear algebra is to classify fields with respect to Witt equivalence, that is equivalence defined by isomorphism of their Witt rings of symmetric bilinear forms (which also includes fields of characteristic two). The classification problem  turns out to be manageable only when restricted to some specific classes of fields and, in fact, is completely resolved only in a few rather special cases. Trivial examples include quadratically closed fields, which are all Witt equivalent, their Witt ring being just $\Z/2\Z$, and real closed fields, which, again, are all Witt equivalent, their Witt ring being $\Z$. It is also not hard to see that finite fields of characteristic different from two are Witt equivalent if and only if they are of the same level, and that there are only two kinds of nonisomorphic Witt rings for this class of fields: $\Z/4\Z$ for fields with the number of elements congruent to $3 (\operatorname{mod} 4)$, and the group ring $\Z/2\Z[F^*/{F^*}^2]$ for the remaining finite fields $F$, $\operatorname{char} F \neq 2$. 

A slightly more involved, but still to be found in graduate-level textbooks, is the case of local fields, that is complete discrete valued fields with finite residue fields. Recall that archimedean local fields are either $\R$ or $\C$, non-archimedean local fields of characterictic 0 are finite extensions of $\Q_p$ (which include the dyadic case of extensions of $\Q_2$), and non-archimedean local fields of characterictic $p$ are finite extensions of $\mathbb{F}_q((x))$, $q = p^n$, for some $n \in \mathbb{N}$. Each non-archimedean nondyadic local field $F$ is Witt equivalent to either $\Q_3$, or to $\Q_5$, depending on whether the number of elements of its residue field is congruent to $3 (\operatorname{mod} 4)$, or to $1 (\operatorname{mod} 4)$. For a dyadic local field $F$, that is an extension of $\Q_2$ of degree $n$, if $n$ is odd, then the Witt equivalence class of $F$ depends only on $n$, and if $n$ is even, then there are exactly two Witt equivalence classes, one with $\sqrt{-1} \in F$, and one with $\sqrt{-1} \notin F$. These results are to be found, for example, in Lam's classical monograph \cite{l}. 

Witt equivalence preserves the order of the square class group of the field, and so classification of fields up to Witt equivalence reduces to classification of fields with a given cardinality of the group of square classes. Classification of fields having only finitely many square classes was in the scope of interest of a large part of the quadratic forms community in the late 1970s and early 1980s. Denote by $q$ the number of square classes of a field $F$, and by $w(q)$ the number of classes of Witt equivalent fields of characteristic $\neq 2$ with $q$ square classes. The exact values of $w(q)$ are known only for $q \leq 32$, were computed in a series of papers by Carson and Marshall \cite{cm}, Cordes \cite{cor73}, Kula \cite{kula81}, Kula, Szczepanik and Szymiczek \cite{kss}, Szczepanik \cite{szcz85}, and Szymiczek \cite{szym75}, and are summarized in the following table:
$$\begin{array}{|c||c|c|c|c|c|c|}
\hline
q & 1 & 2 & 4 & 8 & 16 & 32 \\
\hline \hline
w(q) & 1 & 3 & 6 & 17 & 51 & 155 \\
\hline
\end{array}$$

The classification of global fields with respect to Witt equivalence proved to be significantly more challenging than the classification of local fields. Recall that a global field is either a number field, i.e. a finite extension of $\Q$, or a function field in one variable over a finite field. Since completions of global fields at their primes are local fields, Witt equivalence of completions of global fields is well-understood. Witt equivalence of global fields was completely resolved by a remarkable local-global principle, whose three different proofs were given by Perlis, Szymiczek, Conner, Litherland \cite{Petal}, and Szymiczek \cite{Sz1}, \cite{szym97}, which states that two global fields of characteristic $\neq 2$ are Witt equivalent if and only if their primes can be paired so that corresponding completions are Witt equivalent. Moreover, Baeza and Moresi \cite{bm} showed that any two global fields of characteristic 2 are Witt equivalent, and it is not difficult to see that a global field of characteristic 2 is never Witt equivalent to a global field of characteristic different from 2. As a consequence of the local-global principle, it is also possible to provide a complete list of invariants of Witt equivalence for number fields, as shown by Carpenter \cite{C}.

For fields $K$ and $k$, recall that $K$ is a function field over $k$, if $K$ is a finitely generated field extension of $k$. If $\operatorname{trdeg}(K:k)=n$ we say $K$ is a  function field in $n$ variables  over $k$. The field of constants of $K$ over $k$ is the algebraic closure of $k$ in $K$; it is a finite extension of $k$ \cite[Chapter 10, Proposition 3]{sl}, so that, in general, we do not require that $k$ is the field of constants of $K$ over $k$. Two more classes of fields where Witt equivalence is well-understood are function fields in one variable over algebraically closed fields of characteristic $\neq 2$, and function fields in one variable over real closed fields. In the former case, it is implied by an old result by Tsen \cite{tsen1933}, that two function fields are Witt equivalent if and only if their underlying algebraically closed fields are of the same cardinality, whilst in the latter one, necessary and sufficient conditions for Witt equivalence were given by Koprowski \cite{kop2} and Grenier-Boley and Hoffmann (with appendix by Scheiderer) \cite{gh}.

Given the above historical background, it is natural to ask for criteria of Witt equivalence of function fields over global and local fields. This question was first addressed by Koprowski in \cite{kop1}. Unfortunately, there appears to be a serious error in the proof of the main theorem in \cite{kop1}. 

The authors of the present paper considered Witt equivalence of function fields over global fields in their earlier work \cite{gm2016}. In the spirit of the local-global principle by Perlis, Szymiczek, Conner, and Litherland, that established a bijection between valuations of two Witt equivalent global fields of characteristic $\neq 2$, the authors managed to show that Witt equivalence of two function fields over global fields induces in a canonical way a bijection $v \leftrightarrow w$ between Abhyankar valuations $v$ of $K$ having residue field not finite of characteristic 2 and Abhyankar valuations $w$ of $L$ having residue field not finite of characteristic 2 (see \cite[Theorem 7.5]{gm2016}). The main tool for setting up this bijection was a method of constructing valuations described in \cite{aej}, which is based, in turn, on earlier constructions, of a similar sort, described in \cite{j0} and \cite{wa}. These results were then applied to study Witt equivalence of function fields of curves defined over global fields, that is function fields in one variable over global fields; see \cite[Corollary 8.1 and Corollary 8.2]{gm2016}.

In the present work the authors extend the results on curves of \cite{gm2016} to function fields of curves defined over local fields. The main result of this article is Theorem \ref{local one variable}, which states that Witt equivalence of two function fields in one variable over local fields of characteristic $\neq 2$  induces a canonical bijection between certain subsets of Abhyankar valuations of the corresponding fields. Contrary to the intuition that one might have developed based on the necessary and sufficient conditions for Witt equivalence of local and global fields, the case of function fields of curves over local fields (Theorem \ref{local one variable}) is in no way easier to settle than the case of function fields of curves over global fields (\cite[Corollary 8.1]{gm2016}). Theorem \ref{local one variable} is then applied to show that, under certain assumptions, Witt equivalence of two function fields of curves over local fields $k$ and $\ell$ implies Witt equivalence of $k$ and $\ell$. This extends \cite[Corollary 8.2]{gm2016} to the local case.

The main new results of the paper are found in Section 3. Throughout the entire exposition the authors make extensive use of the notion of hyperfields, which seem to provide a natural and convenient language to study Witt equivalence. In Section 2 we recall basic terminology, establish fundamental connections between hyperfields and valuations, apply the results of \cite{aej} to understand the behavior of valuations under Witt equivalence, recall the terminology of function fields and Abhyankar valuations, as well as of nominal transcendence degree. All of this is a summary of the material presented in Sections 2--6 of \cite{gm2016}, and we refer to our earlier paper as far as background and most of the details are concerned.

\section{Background and preliminaries}

The main new tool used to study Witt equivalence here are hyperfields. They are objects like fields, but with addition allowed to be multivalued, and were introduced in 1956 by Krasner \cite{kr0} and used for the approximation of valued fields. For the decades that followed, structures with multivalued addition have been better known to computer scientists, due to their applications to fuzzy logic, automata, cryptography, coding theory and hypergraphs (see \cite{cors03}, \cite{dav07} and \cite{zies06}), as well as, to some extent, to mathematicians with expertise in harmonic analysis (see \cite{litv11}). 

Recently, the hyperstructure theory has witnessed a certain revival in connection with various fields: in a series of papers by Connes and Consani \cite{cc-1}, \cite{cc-2}, \cite{cc-3}, with applications to number theory, incidence geometry, and geometry in characteristic one, in works by Viro \cite{viro-1}, \cite{viro-2}, with applications to tropical geometry, by Izhakian and Rowen \cite{ir} and Izhakian, Knebusch and Rowen \cite{ikr}, with applications to recently introduced algebraic objects such as supertropical algebras, or by Lorscheid \cite{lor-1}, \cite{lor-2} to blueprints -- these are algebraic objects which aim to provide a firm algebraic foundation to tropical geometry.
Jun applied the idea of hyperstructures to generalize the definition of valuations and developed the basic notions of algebraic geometry over hyperrings \cite{jun-1}, \cite{jun-2}, \cite{jun-3}, and, finally, the second author introduced independently the notion of hyperfields (there called multifields) to study quadratic forms in \cite{M2} -- these ideas were later applied, for example, to develop the theory of orderings of higher level for hyperstructures \cite{gla2010}, \cite{gm2012}.

Hyperfields provide a convenient and very natural way to describe Witt equivalence, as described in detail in \cite{gm2016}. In what follows we shall review the basic concepts and definitions used later in our paper. We point the reader to Sections 2--6 of \cite{gm2016} as far as background and proofs are concerned.

Just as remarked above, a \it hyperfield \rm is a system $(H, +, \cdot, -, 0, 1)$ where $H$ is a set, $+$ is a multivalued binary operation on $H$, i.e., a function from $H\times H$ to the set of all subsets of $H$, $\cdot$ is a binary operation on $H$, $- : H \rightarrow H$ is a function, and $0,1$ are elements of $H$ such that

\begin{enumerate}
\item[(I)] $(H,+,-,0)$ is a {\it canonical hypergroup}, terminology as in Mittas \cite{mi}, i.e. for all $a,b,c \in H$,
\begin{enumerate}
\item[(1)] $c\in a+b$ $\Rightarrow$ $a \in c+(-b)$,
\item[(2)] $a \in b+0$ iff $a=b$,
\item[(3)]$(a+b)+c = a+(b+c)$,
\item[(4)] $a+b= b+a$;
\end{enumerate}
\item[(II)] $(H,\cdot, 1)$ is a commutative monoid, i.e. for all $a,b,c \in H$, 
\begin{enumerate}
\item[(1)] $(ab)c=a(bc)$, 
\item[(2)]$ab=ba$,
\item[(3)] $a1=a$;
\end{enumerate}
\item[(III)] $a0 = 0$ for all $a\in H$;
\item[(IV)] $a(b+c) \subseteq ab+ac$;
\item[(V)] $1\ne 0$ and every non-zero element has a multiplicative inverse.
\end{enumerate}

Hyperfields are made into a category be declaring a morphism from $H_1$ to $H_2$, where $H_1$, $H_2$ are hyperfields, to be a function $\alpha : H_1 \rightarrow H_2$ which satisfies $\alpha(a+b) \subseteq \alpha(a)+\alpha(b)$, $\alpha(ab) = \alpha(a)\alpha(b)$, $\alpha(-a)=-\alpha(a)$, $\alpha(0)=0$, $\alpha(1)=1$.

{\it Quotient hyperfields}, as defined below, will be of constant use in what follows. For a subgroup $T$ of $H^*$, where $H$ is a hyperfield, denote by $H/_m T$ the set of equivalence classes with respect to the equivalence relation $\sim$ on $H$ defined by 
$$a\sim b \mbox{ if and only if } as = bt \mbox{ for some } s,t\in T.$$ 
The operations on $H/_m T$ are the obvious ones induced by the corresponding operations on $H$: denoting by $\overline{a}$ the equivalence class of $a$ set $\overline{a}\overline{b} = \overline{ab}$, $-\overline{a} = \overline{-a}$, $0 = \overline{0}$, $1 = \overline{1}$, and
$$\overline{a} \in \overline{b}+\overline{c} \mbox{ if and only if } as \in bt+cu \mbox{ for some }s,t,u \in T.$$
$(H/_m T, +,\cdot, -, 0,1)$ is then a hyperfield that we shall refer to as quotient hyperfield. The group of non-zero elements of $H/_m T$ is $H^*/T$.

For a hyperfield $H = (H, +,\cdot, -, 0,1)$ the \it prime addition \rm on $H$ is defined by
\begin{linenomath}\[a+'b = \begin{cases} a+b, &\text{ if one of } a,b \text{ is zero }, \\ a+b\cup \{ a,b\}, &\text{ if } a \ne 0, \ b\ne 0,\ b \ne -a, \\ H, &\text{ if } a \ne 0, \ b\ne 0,\ b = -a. \end{cases}\]\end{linenomath}
For any hyperfield $H:=(H, +,\cdot, -, 0,1)$, $H':=(H, +',\cdot, -, 0,1)$ is also a hyperfield \cite[Proposition 2.1]{gm2016}. We shall call $H'$ the \it prime \rm of the hyperfield $H$. One easily verifies that if $T$ is a subgroup of $H^*$ then $H'/_mT = (H/_mT)'$.

Let $K$ be a field. The \it quadratic hyperfield \rm of $K$, denoted $Q(K)$, is defined to be the prime of the hyperfield $K/_m K^{*2}$. The interest in $Q(K)$ stems from its connection to the algebraic theory of quadratic forms. Denote by $W(K)$ the Witt ring of non-degenerate symmetric bilinear forms over $K$; see \cite{l}, \cite{M1} or \cite{w} for the definition in case $\operatorname{char}(K) \ne 2$ and \cite{KR}, \cite{krw} or \cite{mh} for the definition in the general case. A (non-degenerate diagonal) {\it binary form} over $K$ is an ordered pair $\langle \overline{a},\overline{b}\rangle$, $\overline{a}, \overline{b} \in K^*/K^{*2}$. The \it value set \rm of such a form, denoted by  $D_K\langle \overline{a},\overline{b} \rangle$, is the set of non-zero elements of $\overline{a}+\overline{b}$, i.e., $D_K\langle \overline{a},\overline{b}\rangle$ is the image under $K^* \rightarrow K^*/K^{*2}$ of the subset $D_K\langle a,b\rangle$ of $K^*$ defined by 
\begin{linenomath}\[D_K\langle a,b\rangle := \begin{cases} K^*, & \text{ if } -ab \in K^{*2}, \\ \{ z\in K^* : z = ax^2+by^2, x,y \in K\}, &\text{ otherwise.}\end{cases}\]\end{linenomath}
Two binary forms $\langle \overline{a},\overline{b}\rangle$ and $\langle \overline{c},\overline{d}\rangle$
are considered to be \it equivalent, \rm denoted $\langle \overline{a},\overline{b}\rangle \approx \langle \overline{c},\overline{d}\rangle$, if $\overline{c}\in D_K\langle \overline{a},\overline{b}\rangle$ and $\overline{a}\overline{b} = \overline{c}\overline{d}$.

A hyperfield isomorphism $\alpha$ between two quadratic hyperfields $Q(K)$ and $Q(L)$, where $K,L$ are fields, can be viewed as a group isomorphism $\alpha : K^*/K^{*2} \rightarrow L^*/L^{*2}$ such that $\alpha(-\overline{1}) = -\overline{1}$ and \begin{linenomath}\[\alpha(D_K\langle \overline{a},\overline{b}\rangle) = D_L\langle \alpha(\overline{a}),\alpha(\overline{b})\rangle \text{ for all } \overline{a},\overline{b} \in K^*/K^{*2}.\]\end{linenomath}
We say two fields $K$ and $L$ are \it Witt equivalent, \rm denoted $K\sim L$, to indicate that $W(K)$ and  $W(L)$ are isomorphic as rings.
Combining the results in  \cite{bm}, \cite{h}, and \cite{M1} one gets that $K\sim L$ iff $Q(K)$ and $Q(L)$ are isomorphic as hyperfields (for details, see \cite[Proposition 3.2]{gm2016}).

For two hyperfields $H_1, H_2$ a morphism $\iota: H_1 \rightarrow H_2$ induces a morphism $\overline{\iota} : H_1/_m \Delta \rightarrow H_2$ where $\Delta: = \{ x\in H_1^* : \iota(x)=1\}$. The morphism $\iota$ is said to be a \it quotient morphism \rm if $\overline{\iota}$ is an isomorphism, or, equivalently, if $\iota$ is surjective, and 
$$\iota(c) \in \iota(a)+\iota(b) \mbox{ if and only if } cs\in at+bu \mbox{ for some } s,t,u \in \Delta.$$  
A morphism $\iota : H_1 \rightarrow H_2$ is said to be a \it group extension \rm if $\iota$ is injective,
every $x\in H_2^* \backslash \iota(H_1^*)$ is \it rigid \rm, meaning $1+x \subseteq \{ 1,x\}$, and 
$$\iota(1+y)=1+\iota(y), \mbox{ for all } y \in H_1, y\ne -1.$$

For a field $K$ we adopt the standard notation from valuation theory: if $v$ is a valuation on $K$, $\Gamma_v$ denotes the value group, $A_v$ the valuation ring, $M_v$ the maximal ideal, $U_v$ the unit group, $K_v$ the residue field, and $\pi = \pi_v: A_v \rightarrow K_v$ the canonical homomorphism, i.e., $\pi(a) = a+M_v$. We say $v$ is \it discrete rank one \rm if $\Gamma_v = \mathbb{Z}$. 

We will be interested in the subgroup $T= (1+M_v)K^{*2}$ of $K^*$. Consider the canonical group isomorphism $\alpha : U_vK^{*2}/(1+M_v)K^{*2} \rightarrow K_v^*/K_v^{*2}$ induced by 
$$x \in U_v \mapsto \pi(x) \in K_v^*.$$ 
Define $\iota : Q(K_v) \rightarrow K/_m T$ by $\iota(0) = 0$ and $\iota(a) = \alpha^{-1}(a)$ for $a \in K_v^*/K_v^{*2}$. If $v$ is non-trivial, then, by a variant of Springer's Theorem \cite{sp1}, \cite{sp2}, the map $Q(K) \rightarrow K/_mT$ defined by $\overline{x} \mapsto xT$  is a quotient morphism and $\iota$ is a group extension (see \cite[Propositions 4.1, 4.2]{gm2016}). The cokernel of the group embedding $\alpha^{-1} : K_v^*/K_v^{*2} \rightarrow K^*/T$ is equal to $K^*/U_vK^{*2} \cong \Gamma_v/2\Gamma_v$. For this reason we sometimes say that $K/_mT$ is a \it group extension of $Q(K_v)$ by the group \rm $\Gamma_v/2\Gamma_v$. If $v$ is non-trivial and $\operatorname{char}(K_v) \ne 2$, then $K/_m T$ is naturally identified with $Q(\tilde{K}_v)$, where $\tilde{K}_v$ denotes the henselization of $(K,v)$ \cite[Proposition 4.3]{gm2016}.

If $v,v'$ are valuations on $K$ with $v \preceq v'$, i.e. such that $v'$ is a coarsening of $v$, meaning $A_v \subseteq A_{v'}$, then $M_{v'} \subseteq M_v$, and, consequently, $(1+M_{v'})K^{*2} \subseteq (1+M_v)K^{*2}$. Denote by $\overline{v}$ the valuation on $K_{v'}$ induced by $v$, that is
$$\overline{v}(\pi_{v'}(a)) = v(a), \mbox{ for } a\in U_{v'}.$$
The valuations $\overline{v}$ and $v$ have the same residue field. If $v$ and $v'$ are non-trivial and $v'$ is a proper coarsening of $v$, meaning $A_v \subsetneq A_{v'}$, then $K/_m (1+M_v)K^{*2}$ is a group extension of the hyperfield $K_{v'}/_m(1+M_{\overline{v}})K_{v'}^{*2}$ in a natural way,
and the following diagram of hyperfields and hyperfield morphisms is commutative:
\begin{linenomath}\[\xymatrix{
Q(K) \ar[r] & K/_m(1+M_{v'})K^{*2} \ar[r] & K/_m(1+M_v)K^{*2} \\
& Q(K_{v'}) \ar[u] \ar[r] & K_{v'}/_m (1+M_{\overline{v}})K_{v'}^{*2} \ar[u]\\
& & Q(K_v) \ar[u]
}\]\end{linenomath}
Here, the horizontal arrows are quotient morphisms and the vertical arrows are group extensions.

Let $T$ be a subgroup of $K^*$.
Adopting the well-known terminology from the algebraic theory of quadratic forms, we say that $x\in K^*$ is \it $T$-rigid \rm if $T+Tx \subseteq T\cup Tx$, and denoting by 
\begin{linenomath}\[B(T) := \{ x \in K^*: \text{ either } x \text{ or } -x \text{ is not } T\text{-rigid}\}\]\end{linenomath}
we will refer to the elements of $B(T)$ as to the \it $T$-basic \rm elements.
Note that if $x\in K^*$ is $T$-rigid and
$y=tx$, for some $t\in T$, then $y$ is $T$-rigid. Consequently, $B(T)$ is a union of cosets of $T$.
The element $-1$ is not $T$-rigid, because $0\in T-T$, so $\pm T \subseteq B(T)$. If, in fact, $\pm T = B(T)$, and either $-1\in T$ or $T$ is additively closed, we shall say that the subgroup $T$ is {\it exceptional.}
If $H \subseteq K^*$ is a subgroup containing $B(T)$, then there exists a subgroup $\hat{H}$ of $K^*$ such that $H \subseteq \hat{H}$ and $(\hat{H}:H) \le 2$, and a valuation $v$ of $K$ such that $1+M_v \subseteq T$ and $U_v \subseteq \hat{H}$. Moreover, $\hat{H}$ can be taken to be simply $H$, unless $T$ is exceptional  \cite[Theorem 2.16]{aej}. By \cite[Theorem 5.18]{M1} $B(K^{*2})$ is a subgroup of $K^*$, and in the case when $T = (1+M_v)K^{*2}$, for some non-trivial valuation $v$ of $K$, $B(T) \subseteq U_vK^{*2}$ and \begin{linenomath}\[B(T) = \{ x\in K^* : \overline{x} = \iota(\overline{y}) \text{ for some } y \in B(K_v^{*2})\},\]\end{linenomath} where $\iota : Q(K_v) \hookrightarrow K/_mT$ is the morphism described above.  $B(T)$ is a group and the group isomorphism $\iota : K_v^*/K_v^{*2} \rightarrow U_vK^{*2}/T$ induces a group isomorphism $B(K_v^{*2})/K_v^{*2} \rightarrow B(T)/T$. $T$ is exceptional if and only if $K_v^{*2}$ is exceptional. See \cite[Proposition 4.6]{gm2016} for details.

We will make frequent use of the following result:

\begin{theorem}[{\cite[Theorem 5.3]{gm2016}}] \label{main lemma} Suppose $K$, $L$ are fields, $\alpha : Q(K) \rightarrow Q(L)$ is a hyperfield isomorphism and $v$ is a valuation on $K$ such that $\Gamma_v$ is finitely generated as an abelian group. Suppose either
(i) the basic part of $(1+M_v)K^{*2}$ is $U_vK^{*2}$ and $(1+M_v)K^{*2}$ is unexceptional, or
(ii) the basic part of $(1+M_v)K^{*2}$ is $(1+M_v)K^{*2}$ and $(1+M_v)K^{*2}$ has index 2 in $U_vK^{*2}$.
Then there exists a valuation $w$ on $L$ such that the image of $(1+M_v)K^{*2}/K^{*2}$ under $\alpha$ is $(1+M_w)L^{*2}/L^{*2}$ and $(L^*:U_wL^{*2}) \ge (K^*:U_vK^{*2})$. If (i) holds, then
the image of $U_vK^{*2}/K^{*2}$ under $\alpha$ is $U_wL^{*2}/L^{*2}$.
\end{theorem}

If $K$, $L$ are fields and $\alpha : Q(K) \rightarrow Q(L)$ is a hyperfield isomorphism such that the image of $(1+M_v)K^{*2}/K^{*2}$ under $\alpha$ is $(1+M_w)L^{*2}/L^{*2}$, then $\alpha$ induces a hyperfield isomorphism $K/_m (1+M_v)K^{*2} \rightarrow L/_m (1+M_w)L^{*2}$ such that the obvious diagram
\begin{linenomath}\begin{equation}\label{d1} \xymatrix{
Q(K) \ar[r] \ar[d] & Q(L) \ar[d] \\
K/_m(1+M_v)K^{*2} \ar[r] & L/_m(1+M_w)L^{*2}}\end{equation}\end{linenomath}
commutes. If, in addition, the image of $U_vK^{*2}/K^{*2}$ under $\alpha$ is $U_wL^{*2}/L^{*2}$, then $\alpha$ induces a hyperfield isomorphism $Q(K_v)\rightarrow Q(L_w)$ and a group isomorphism $\Gamma_v/2\Gamma_v \rightarrow \Gamma_w/2\Gamma_w$ such that the obvious diagrams
\begin{linenomath}\begin{equation}\label{d2} \xymatrix{
K/_m(1+M_v)K^{*2} \ar[r] & L/_m(1+M_w)L^{*2} \\
 Q(K_v) \ar[u] \ar[r]  & Q(L_w) \ar[u]
}\end{equation}\end{linenomath}
and
\begin{linenomath}\begin{equation}\label{d3} \xymatrix{
Q(K)^* \ar[r] \ar[d] & Q(L)^* \ar[d] \\
\Gamma_v/2\Gamma_v \ar[r] & \Gamma_w/2\Gamma_w
}\end{equation}\end{linenomath}
commute. We are assuming here that $v,w$ are non-trivial. See \cite[Proposition 5.4]{gm2016} for details.

The \it nominal transcendence degree \rm of $K$ is defined to be
\begin{linenomath}\[\operatorname{ntd}(K) := \begin{cases} \operatorname{trdeg}(K:\mathbb{Q}) &\text{ if } \operatorname{char}(K) = 0 \\ \operatorname{trdeg}(K:\mathbb{F}_p)-1 &\text{ if } \operatorname{char}(K) = p\ne 0\end{cases}.\]\end{linenomath}

For any abelian group $\Gamma$, the \it rational rank \rm of $\Gamma$, denoted $\operatorname{rk}_{\mathbb{Q}}(\Gamma)$, is defined to be the dimension of the $\mathbb{Q}$-vector space $\Gamma \otimes_{\mathbb{Z}} \mathbb{Q}$. If $K$ is a function field over $k$ and $v$ is a valuation on $K$, the \it Abhyankar inequality \rm asserts that \begin{linenomath}\[\operatorname{trdeg}(K:k) \ge \operatorname{rk}_{\mathbb{Q}}(\Gamma_v/\Gamma_{v|k}) + \operatorname{trdeg}(K_v: k_{v|k}),\]\end{linenomath} where $v|k$ denotes the restriction of $v$ to $k$.  We will say the valuation $v$ is \it Abhyankar \rm (relative to $k$)  if \begin{linenomath}\[\operatorname{trdeg}(K:k) = \operatorname{rk}_{\mathbb{Q}}(\Gamma_v/\Gamma_{v|k}) + \operatorname{trdeg}(K_v: k_{v|k}).\]\end{linenomath}
In this case it is well known that $\Gamma_v/\Gamma_{v|k}$ is finitely generated
and $K_v$ is a function field over $k_{v|k}$ \cite[Corollary 26]{fvk}.

\section{Function fields in one variable over local fields}

As remarked in the Introduction, Witt equivalence of local fields is well understood. If $k$ is a local field then \begin{linenomath}\[|k^*/k^{*2}| = \begin{cases} 1 \text{ if } k = \mathbb{C} \\ 2 \text{ if } k = \mathbb{R} \\ 4 \text{ if } k \text{ is } p\text{-adic}, p \ne 2 \\ 2^{[k:\hat{\mathbb{Q}}_2]+2} \text{ if } k \text{ is dyadic} \\ \infty \text{ if } \operatorname{char}(k) = 2 \end{cases}. \]\end{linenomath} If $k$ is $p$-adic and $\ell$ is $p'$-adic, $p,p' \ne 2$, then $k \sim \ell$ iff $-1 \in k^{*2}$ $\Leftrightarrow$ $-1 \in \ell^{*2}$. If $k$ and $\ell$ are dyadic, then $k\sim \ell$ iff either $[k:\hat{\mathbb{Q}}_2] = [\ell:\hat{\mathbb{Q}}_2]$ is odd or $[k:\hat{\mathbb{Q}}_2] = [\ell:\hat{\mathbb{Q}}_2]$ is even and $-1 \in k^{*2}$ $\Leftrightarrow$ $-1 \in \ell^{*2}$.
If $\operatorname{char}(k)=2$ then $k = \mathbb{F}_{2^s}((x))$ for some $s\ge 1$. It follows that $k^2 = \mathbb{F}_{2^s}((x^2))$ so $[k:k^2] = 2$ and $1,x$ is a basis of $k$ as a vector space over $k^2$. It follows from this and \cite[Proposition 2.10]{bm} that if $k,\ell$ are local fields of characteristic $2$ then $k \sim \ell$.

If $k$ is a non-archimedean local field, we denote by $v_0$ the unique discrete rank one valuation of $k$.

The following two lemmas appear to be well-known. Although we do not know the precise reference, these results are obviously closely related to results in \cite{fks} and \cite{v}.

\begin{lemma} \label{weird valuations 1} Suppose $w$ is a non-trivial valuation on $k \in \{\mathbb{R}, \mathbb{C}\}$. Then $\Gamma_w$ is divisible and $k_w$ is algebraically closed.
\end{lemma}

\begin{proof} Case 1: Suppose $k = \mathbb{C}$. Suppose $x\in k^*$. For each integer $n\ge 1$ $\exists$ $y \in k^*$ such that $y^n=x$. Thus $w(x) = w(y^n) = nw(y)$. This proves $\Gamma_w$ is divisible.  Suppose $f(x) = x^n+a_{n-1}x^{n-1}+\dots +a_0$ where $a_i\in A_w$, $i=0,\dots,n-1$. Clearly $\exists$ $r \in k$, $f(r)=0$. Since $A_w$ is integrally closed, $r \in A_w$. Thus the image of $r$ in $k_w$ is a root of the image of $f(x)$ in $k_w[x]$. This proves that $k_w$ is algebraically closed. 

Case 2: Suppose $k = \mathbb{R}$. Suppose $x\in k^*$. One of $x$, $-x$ is positive. Also $w(x)=w(-x)$. Thus, replacing $x$ by $-x$ if necessary, we can assume $x>0$. For each integer $n\ge 1$ $\exists$ $y \in k^*$ such that $y^n= x$. Thus $w(x) = w(y^n) = nw(y)$. This proves $\Gamma_w$ is divisible. Suppose $f(x) = x^n+a_{n-1}x^{n-1}+\dots +a_0$, $a_i\in A_w$, $i=0,\dots,n-1$. By the approximation theorem for $V$-topologies $\exists$ $b_0 \in k$ such that $w(b_0-a_0)>0$ and $|b_0+1| <1$ (so $b_0<0$). Let $b_i = a_i$ for $i=1,\dots,n-1$. Clearly $g(x) := x^n+b_{n-1}x^{n-1}+\dots +b_0$ has a root $r\in k$. Since $A_w$ is integrally closed, $r \in A_w$. By construction $f(x)$ and $g(x)$ have the same image in $k_w[x]$. Thus the image of $r$ in $k_w$ is a root of the image of $f(x)$ in $k_w[x]$. This proves that $k_w$ is algebraically closed.
\end{proof}

\begin{lemma} \label{weird valuations 2} Suppose $w$ is  a non-trivial valuation on a non-archimedean local field $k$. Then either (1) $w=v_0$, or (2) $w, v_0$ are independent. In case (2), $k_w$ is algebraically closed and $\Gamma_w$ is $p$-divisible for each prime $p \ne \operatorname{char}(k_v)$. If $\operatorname{char}(k)=0$ then $\Gamma_w$ is also $p$-divisible for $p = \operatorname{char}(k_{v_0})$.
\end{lemma}

\begin{proof}  Denote by $u$ the finest common coarsening of $v_0$ and $w$. Since $v_0$ is discrete rank 1, $u = v_0$ or $u$ is trivial. If $u = v_0$ then $w \preceq v_0$. Since the residue field $k_v$ is finite, the pushdown of $w$ to $k_{v_0}$ must be the trivial valuation, so $w=v_0$ in this case. If $u$ is trivial then $w$ and $v_0$ are obviously independent. Suppose now that $v_0,w$ are independent. Let $x\in k^*$. Suppose first that $p\ne \operatorname{char}(k_{v_0})$. By the approximation theorem $\exists$ $y\in k^*$ such that $w(y-x)>w(x)$ and $v_0(y-1)>0$. By Hensel's lemma $\exists$ $z \in k^*$, $z^p=y$. Since $w(y-x)>w(x)$, $w(x)=w(y)$. Then $w(x)=w(y) = w(z^p)=pw(z)$. Suppose now that $\operatorname{char}(k)=0$ and $p\ne \operatorname{char}(k_{v_0})$. In this case pick $y \in k^*$ so that $w(y-x)>w(x)$ and $v_0(y-1) > \frac{ep}{p-1}$ where $e=v_0(p)$. Here, we are identifying $\Gamma_v = \mathbb{Z}$. Then again, $\exists$ $z\in k^*$, $z^p=y$, and one can argue as before. This proves the assertion concerning $\Gamma_w$. Suppose $f(x) = x^n+a_{n-1}x^{n-1}+\dots +a_0$ where $a_i\in A_w$, $i=0,\dots,n-1$. Let $p = \operatorname{char}(k_{v_0})$. Case 1: $n$ not divisible by $p$. In this case choose $b_i$, $i=0,\dots,n-1$ so that $w(b_i-a_i)>0$, $i=0,\dots,n-1$, $v_0(b_0+1)>0$, and $v_0(b_i)>0$, $i=1,\dots,n-1$. By Hensel's lemma, $g(x) = x^n+b_{n-1}x^{n-1}+\dots +b_0$ has a root $r \in k$. Then $r\in A_w$ and the image of $r$ in $k_w$ is a root of the image of $f(x)$ in $k_w[x]$. Case 2: $n$ is divisible by $p$. In this case choose $b_i$, $i=0,\dots,n-1$, so that $w(b_i-a_i)>0$, $i=0,\dots,n-1$, $v_0(b_1+1)>0$, and $v_0(b_i)>0$, $i=0,\dots,n-1$, $i\ne 1$ and argue as before.
\end{proof}

The following result, which is an extension of \cite[Lemma 7.2]{gm2016}, appears to be well-known, but, as we are not able to point the reader to a reference, we decided to include a proof.

\begin{lemma} \label{well known two} Suppose $K$ is a function field in one variable over $k$, where $k$ is a local field of characteristic $\ne 2$, or any real or algebraically closed field. Then

(1) There are infinitely many discrete rank one Abhyankar valuations on $K$ trivial on $k$.

(2) The group $K^*/K^{*2}$ is infinite.

(3) $B(K^{*2}) = K^*$.
\end{lemma}

\begin{proof} (1) This is clear if $K$ is real or algebraically closed. Otherwise, there exists a subfield $K_0 \subseteq K$ with $\operatorname{ntd}(K_0) = \operatorname{ntd}(K)-1$. Fix $x\in K$ transcendental over $K_0$. $K$ is a finite extension of $K_0(x)$. The principal ideal domain $K_0[x]$ has infinitely many irreducibles. Each irreducible $f$ of $K_0[x]$ defines a discrete rank one valuation $v_f$ on $K_0(x)$ with residue field $K_0[x]/(f)$. The valuation $v_f$  extends in some (possibly non-unique) way to a discrete rank one valuation on $K$ whose residue field is some finite extension of $K_0[x]/(f)$.  

(2) is true for any field $K$ having infinitely many inequivalent discrete rank one valuations. Let $v_1,\dots,v_n$ be inequivalent discrete rank one valuations on $K$. Use the approximation theorem to produce $x_i\in K^*$, $i=1,\dots,n$ so that $v_i(x_j)=\delta_{ij}$ (Kronecker's delta), for $i,j=1,\dots,n$. Then the $2^n$ products $x_1^{e_1}\dots x_n^{e_n}$, $e_i \in \{ 0,1\}$, belong to distinct square classes. This proves $|K^*/K^{*2}|\ge 2^n$. Since $n$ is can be chosen to be any positive integer, the result follows.

(3) Suppose $x\in K^*$, $x\notin K^{*2}$. Suppose first that $k$ is a $p$-adic local field. Let $v$ be a discrete rank one Abhyankar valuation on $K$ with $v|k$ trivial. Suppose first that $x \in (1+M_v)K^{*2}$, say $x = uc^2$, $u\in 1+M_v$, $c\in K^*$. Since $K_v$ is a finite extension of $k$ it is also a $p$-adic local field. Consequently, there exists $\pi(z) \in K_v^*$ and $\pi(d), \pi(e) \in K_v^*$ such that $\pi(z) = \pi(d)^2+\pi(e)^2$, $\pi(z) \notin K_v^{*2}$. Take $y = c^2(d^2+ue^2) =(cd)^2+xe^2$. Then $y\notin K^{*2}\cup xK^{*2}$. If such a valuation $v$ does not exist, then there exist inequivalent discrete rank one valuations $v,w$ on $K$ with $x \notin (1+M_v)K^{*2}$, $x\notin (1+M_w)K^{*2}$. In this case, use the approximation theorem to choose $a \in K^*$ so that $v(a^2)>v(x)$, $w(a^2)<w(x)$. 
Define $y= a^2+x$. Then $y = x(1+\frac{a^2}{x}) \in x(1+M_v)K^{*2}$, so $y \notin K^{*2}$. Similarly, 
$y= a^2(1+\frac{x}{a^2}) \in (1+M_w)K^{*2}$, so $y\notin xK^{*2}$.

Next suppose $k$ is algebraically closed. Then $K$ is a $C_1$-field so every two dimensional form $\langle 1,x\rangle$ is universal, i.e., $x$ is not rigid. Finally, suppose $k$ is real closed. The argument in \cite[Corollary 3.8]{gh} shows that $u(K) = 2$ and either $K$ is non-real or $K \sim k(t)$. If $K$ is non-real then every two dimensional form $\langle 1,x\rangle$ is universal, i.e., $x$ is not rigid. Suppose $K = k(t)$. We want to show $x$ or $-x$ is not rigid for each $x \in K^*$. We may assume $x \notin K^{*2}$. Modifying $x$ by a square we may assume $x = f(t)\in k[t]$ for some non-constant polynomial $f(t)$. Replacing $x$ by $-x$ if necessary, there exists $p_1,p_2 \in k$ such that $f(p_1)<f(p_2) < 0$. Take $y= r+x$ where $r\in k$, $f(p_1)<-r<f(p_2)$. Then $y \in K^{*2}+xK^{*2}$, $y\notin K^{*2}$ (because $y$ is negative at $p_1$) and $y \notin xK^{*2}$ (because $\frac{y}{x}$ is negative at $p_1$), so $x$ is not rigid.
\end{proof}

Our next result can be viewed as some sort of local analog of \cite[Theorem 7.3]{gm2016}.

\begin{theorem} \label{local zoo} Let $K$ be a function field in one variable over a local field $k$, $\operatorname{char}(k) \ne 2$. In case $k$ is $p$-adic, denote by $v_0$ the canonical valuation on $k$. Let $v$ be a valuation on $K$, and let $T = (1+M_v)K^{*2}$. 
\begin{enumerate}
\item[(I)] Abhyankar Cases:
\begin{enumerate}
\item[Case 1] If$v$ is trivial, then  $T = K^{*2}$, $(K^*:T) = \infty$ and $B(T) = K^* = U_vK^{*2}$.
\item[Case 2] If $\Gamma_v = \mathbb{Z}$ and $v|k$ is trivial, then
\begin{enumerate}
\item[(a)] if $k = \mathbb{R}$ or $\mathbb{C}$ then $(K^*:U_vK^{*2}) = 2$ and $B(T) = \pm T = U_vK^{*2}$,
\item[(b)] if $k$ is $p$-adic, $p\ne 2$, $(K^*:U_vK^{*2}) = 2$, $(U_vK^{*2}: T)=4$ and $B(T) =\pm T$,
\item[(c)] if $k$ is dyadic, $(K^*:U_vK^{*2}) = 2$, $(U_vK^{*2}: T)=2^{[K_v:\hat{\mathbb{Q}}_2]+2}$ and $B(T) = U_vK^{*2}$.
\end{enumerate}
\item[Case 3] If $k$ is $p$-adic and $v|k = v_0$, then
\begin{enumerate}
\item[(a)] if $\Gamma_v = \mathbb{Z} = \Gamma_{v_0}$, $\operatorname{trdeg}(K_v:k_{v_0})=1$, then  $(K^*:U_vK^{*2}) = 2$, $(U_vK^{*2}: T)=\infty$ and $B(T) = U_vK^{*2}$,
\item[(b)] If $\Gamma_v = \mathbb{Z}\times \mathbb{Z} = \Gamma_{v_0}\times \mathbb{Z}$, $\operatorname{trdeg}(K_v:k_{v_0})=0$, then  $(K^*:U_vK^{*2}) = 4$, $(U_vK^{*2}: T)=\begin{cases} 2 \text{ if } p \ne 2\\ 1 \text{ if } p=2\end{cases}$ and $B(T) = \pm T$.
\end{enumerate}
\item[Case 4] If $v|k$ has divisible value group and algebraically closed residue field, then
\begin{enumerate}
\item[(a)] if $\Gamma_v = \Gamma_{v|k}\times \mathbb{Z}$, $\operatorname{trdeg}(K_v : k_{v|k})=0$ then $(K^*:U_vK^{*2}) = 2$ and $B(T) = T = U_vK^{*2}$,
\item[(b)] if $\Gamma_v = \Gamma_{v|k}$, $\operatorname{trdeg}(K_v: k_{v|k})=1$ then $(K^*:T) = \infty$ and $B(T) = K^* = U_vK^{*2}$.
\end{enumerate}
\end{enumerate}
\item[(II)] Non-Abhyankar Cases:
\begin{enumerate}
\item[Case 5] If $k$ is $p$-adic, $v|k = v_0$, $\Gamma_v/\Gamma_{v_0}$ is torsion, and $\operatorname{trdeg}(K_v: k_{v_0})=0$, then $\mathbb{Z} \subseteq \Gamma_v \subseteq \mathbb{Q}$, $(K^*: U_vK^{*2})= 1$ or $2$, $(U_vK^{*2}: T) = 1$ or $2$ and $B(T) = \pm T$.
\item[Case 6] If$v|k$ has divisible value group and algebraically closed residue field, then $\Gamma_v = \Gamma_{v|k}$ is divisible, $K_v = k_{v|k}$ is algebraically closed, and $K^* = U_vK^{*2} = B(T) = T$.
\end{enumerate}
\end{enumerate}
\end{theorem}

\begin{proof} We make use of the isomorphism $K^*/U_vK^{*2} \cong \Gamma_v/2\Gamma_v$, the isomorphism  $U_vK^{*2}/(1+M_v)K^{*2} \cong K_v^*/K_v^{*2}$ and the properties of the $T$-basic part $B(T)$ of $K$ with $T=(1+M_v){K^*}^2$ described above. The arguments are straightforward and follow, more or less, the same line of reasoning as the proof of \cite[Theorem 7.3]{gm2016}. Case 1 uses Lemma \ref{well known two}. Case 2 uses the well-known structure of quadratic hyperfields of local fields. Case 3 (a) uses Lemma \ref{well known two}. Case 3 (b) uses well-known properties of quadratic hyperfields of finite fields. Case 4 uses Lemmas \ref{weird valuations 1} and \ref{weird valuations 2}, and also Lemma \ref{well known two}, in Case 4 (b). Case 5 is more or less trivial. Case 6 uses Lemmas \ref{weird valuations 1} and \ref{weird valuations 2}.
\end{proof}

Let $K$ be a function field in one variable a local field $k$ of characteristic $\ne 2$.
Let $\mu_{K,0}$ be the set of valuations $v$ of $K$ such that $(K^*:U_vK^{*2})=2$, $2^3 \le (U_vK^{*2}: (1+M_v)K^{*2}) < \infty$ and $B((1+M_v)K^{*2}) = U_vK^{*2}$. Let $\mu_{K,1}$ be the set of valuations $v$ on $K$ such that $(K^*:U_vK^{*2})=2$, $(U_vK^{*2}: (1+M_v)K^{*2}) = \infty$ and $B((1+M_v)K^{*2}) = U_vK^{*2}$. Let $\mu_{K,2}$ be the set of valuations $v$ on $K$ such that $(K^*:U_vK^{*2})=4$, $(U_vK^{*2}: (1+M_v)K^{*2}) = 2$ and $B((1+M_v)K^{*2}) = U_vK^{*2}$. Let $\mu_{K,3}$ be the set of valuations $v$ on $K$ such that $(K^*:U_vK^{*2})=4$, $(U_vK^{*2}: (1+M_v)K^{*2}) = 2$ and $B((1+M_v)K^{*2}) = (1+M_v)K^{*2}$.

Elements of $\mu_{K,i}$ correspond to various cases of Theorem \ref{local zoo}:  Case 2 (a) if $i=0$, Case 3 (a) if $i=2$, and Case 3 (b), (subcase $p \ne 2$) if $i=2$ or $3$. Note that \begin{linenomath}\[\mu_{K,0}\cup \mu_{K,1} \cup \mu_{K,2} \cup \mu_{K,3}\]\end{linenomath}  is a subset of the set of all Abhyankar valuations of $K$ over $k$. Of course, some of the sets $\mu_{K,i}$ are empty. Specifically, $\mu_{K,0} \ne \emptyset$ iff $k$ is dyadic, $\mu_{K,1} \ne \emptyset$ iff $k$ is $p$-adic, $\mu_{K,2} \cup \mu_{K,3} \ne \emptyset$ iff $k$ is $p$-adic, $p\ne 2$.

\begin{theorem}\label{local one variable} Suppose $K$, $L$ are function fields in one variable over local fields of characteristic $\ne 2$ which are Witt equivalent via a hyperfield isomorphism $\alpha : Q(K) \rightarrow Q(L)$. Then for each $i\in \{ 0,1,2,3\}$ there is a uniquely defined bijection  between $\mu_{K,i}$ and $\mu_{L,i}$ such that, if $v \leftrightarrow w$ under this bijection, then $\alpha$ maps $(1+M_v)K^{*2}/K^{*2}$ onto $(1+M_w)L^{*2}/L^{*2}$ and $U_vK^{*2}/K^{*2}$ onto $U_wL^{*2}/L^{*2}$ for $i \in \{0,1,2\}$ and such that $\alpha$ maps $(1+M_v)K^{*2}/K^{*2}$ onto $(1+M_w)L^{*2}/L^{*2}$ for $i=3$.
\end{theorem}

\begin{proof} The correspondence $v \leftrightarrow w$ is just the one defined by Theorem \ref{main lemma}.
\end{proof}

\begin{theorem} \label{local rational points} Let $K \sim L$ be function fields in one variable over local fields $k$ and $\ell$ respectively,  with fields of constants $k$ and $\ell$ respectively. Then $k\sim \ell$ except possibly when $k,\ell$ are both dyadic local fields.  In the latter case if there exists $v \in \mu_{K,0}$ with $K_v =k$ and $w \in \mu_{L,0}$ with $L_w= \ell$ then $k \sim \ell$.
\end{theorem}

We will use the following two technical lemmas for the proof, which are also applicable for the proof of the corresponding result for function fields over global fields.

\begin{lemma}[{\cite[Lemma 7.6]{gm2016}}] \label{characteristic 2} Suppose $K$ is a field, $\operatorname{char}(K)=2$, $\overline{x},\overline{y}\in K^*/K^{*2}$, $\overline{x},\overline{y} \ne 1$ and $\overline{y} \in D_K\langle 1,\overline{x}\rangle$. Then $D_K\langle 1,\overline{y}\rangle = D_K\langle 1,\overline{x}\rangle$.
\end{lemma}

\begin{lemma}[{\cite[Lemma 7.10]{gm2016}}] \label{dimension formula}  If $K$ is a function field over a field $k$, $\operatorname{char}(k)=2$, then \begin{linenomath}\[[K:K^2] = 2^{\operatorname{trdeg}(K:k)}\cdot [k:k^2].\]\end{linenomath}
\end{lemma}

We now proceed to the proof of the theorem:

\begin{proof} Suppose first that $k,\ell$ have characteristic different from $2$. Observe that if $k$ is dyadic then $\mu_{K,0} \ne \emptyset$ so $\mu_{L,0} \ne \emptyset$ and $\ell$ is also dyadic. Suppose in addition there exists $v \in \mu_{K,0}$ with $K_v =k$ and $w \in \mu_{L,0}$ with $L_w= \ell$. Let $v \leftrightarrow w$ be the bijection between $\mu_{K,0}$ and $\mu_{L,0}$ defined by Theorem \ref{local one variable}. We know that $K_v \sim L_w$ for any $v,w$ related in this way. Since $K_v$ and $L_w$ are dyadic local fields, this implies $[K_v:\hat{\mathbb{Q}}_2] = [L_w:\hat{\mathbb{Q}}_2]$ for any such $v,w$. We know also that $k \subseteq K_v$ and $\ell \subseteq L_w$. Choosing $v \leftrightarrow w$ so that $[K_v: \hat{\mathbb{Q}}_2] = [L_w:\hat{\mathbb{Q}}_2]$ is minimal, we see that $K_v=k$ and $L_w=\ell$.

Suppose now that $k$ is $p$-adic, $p\ne 2$. Then $\mu_{K,2} \cup \mu_{K,3} \ne \emptyset$ and consequently $\mu_{L,2} \cup \mu_{L,3} \ne \emptyset$, so $\ell$ is $p'$-adic, for some prime $p' \ne 2$. Then $k,\ell$ each has level $1$ or $2$. If $k$ has level $1$ then $K$ and consequently also $L$ has level $1$. Since $\ell$ is algebraically closed in $L$ this implies $\ell$ has level $1$. This proves $k$ and $\ell$ have the same level, so $k\sim \ell$.

Suppose now that $k,\ell \in \{ \mathbb{R}, \mathbb{C}\}$. If $k=\mathbb{C}$, $\ell = \mathbb{R}$, then $K$ has level $1$, $L$ has level $\ge 2$, contradicting $K \sim L$.

Suppose now that $k,\ell$ both have characteristic $2$. Since $[k:k^2] = [\ell: \ell^2] = 2$, Lemma \ref{dimension formula} implies $[K:K^2] = [L:L^2] = 4$ so,
applying \cite[Theorem 2.9]{bm}, we deduce that $K\cong L$. Say $k= \mathbb{F}_{2^s}((x))$, $\ell = \mathbb{F}_{2^{s'}}((x))$. Since the polynomial $f(t) = t^{2^s}+t$ splits in $k$, it splits in $K$, and hence also in $L$. Since $\ell$ is algebraically closed in $L$, $f(t)$ splits in $\ell$, and hence also in $\mathbb{F}_{2^{s'}}$. This proves $s | s'$. By symmetry, $s=s'$, so $k\cong \ell$.

It remains to show that if $\operatorname{char}(k) \ne 2$ and $\operatorname{char}(\ell) = 2$ then $K \not\sim L$. If $k = \mathbb{C}$ then $K$ has $u$-invariant $2$ so $D_K\langle 1,\overline{x}\rangle = K^*/K^{*2}$ for each $x\in K^*$. On the other hand, $[L:L^2] = 4$ and obviously $L+Lx$ is $2$ dimensional over $L^2$ for any $x\in L^*$, $x \notin L^{*2}$. It follows that $D_L\langle 1,\overline{x}\rangle \ne L^*/L^{*2}$ for any $x\in L^*$, $x \notin L^{*2}$. This implies $K \not\sim L$. Suppose now that $k= \mathbb{R}$. If $K$ is not formally real then $K$ has level $2$ and the above argument shows $K \not\sim L$. If $K$ is formally real then obviously $K \not\sim L$. 

Suppose now that $k$ is $p$-adic. Arguing as in \cite[Lemma 7.7]{gm2016} we will show that there exists $\overline{x},\overline{y} \in K^*/K^{*2}$, $\overline{x},\overline{y} \ne 1$ such that $\overline{y}\in D_K\langle 1,\overline{x}\rangle$, $D_K\langle 1,\overline{y}\rangle \not\subseteq D_K\langle 1,\overline{x}\rangle$. Fix inequivalent discrete rank one Abhyankar valuations $v,w$ on $K$ with $\operatorname{char}(K_v), \operatorname{char}(K_w)\ne 2$. Arguing as in the proof of Lemma \ref{well known two} (3) we see that there exist $x \in K^*$, $y \in {K^*}^2 + x{K^*}^2$ with $y \notin {K^*}^2 \cup x{K^*}^2$. It is thus possible to choose $x$ so that $v(x)=w(x)=1$ and $a_0,b_0$ so that $w(a_0)=w(b_0)=0$ and the image of $c = a_0^2+b_0^2$ in the residue field of $w$ is not a square. Define $y=a^2+x$, $z=b^2+y$ (so $z=a^2+b^2+x$) where $a,b$ are such that $v(a)>0$, $w(a-a_0)> 0$, $w(b-b_0)>0$. Then $v(y) = v(x) = 1$, so $x,y\notin K^{*2}$ and $w(a^2+b^2 -c) >0$ so $z= a^2+b^2+x \in c(1+M_w)$. Thus $x,y,z \in K^*$, $\overline{y} \in D_K\langle 1,\overline{x}\rangle$, $\overline{z} \in D_K\langle 1,\overline{y}\rangle$,  $\overline{x}\ne 1$, $\overline{y} \ne 1$. Let $T =  (1+M_w)K^{*2}$. Thus $T +xT = T\cup xT$ and $z \notin T \cup xT$, so $\overline{z} \notin D_K\langle 1,\overline{x}\rangle$. Applying this, along with Lemma \ref{characteristic 2}, we see that $K \not\sim L$ holds also in this case.
\end{proof}


\end{document}